\documentclass{amsart}
\usepackage{amssymb,amsmath,latexsym}

\newtheorem{theorem}{Theorem}[section]
\newtheorem{lemma}[theorem]{Lemma}
\newtheorem{corollary}[theorem]{Corollary}
\theoremstyle{definition}

\theoremstyle{remark}

\numberwithin{equation}{section}

\newcommand{\bea}{\begin{eqnarray}}
\newcommand{\eea}{\end{eqnarray}}
\newcommand{\ba}{\begin{array}}
\newcommand{\ea}{\end{array}}

\input{tcilatex}

\begin{document}
\title[On the fine spectrum]{On the fine spectrum of the second order
difference operator over the sequence spaces\\
$\ell _{p}$\ and $bv_{p},$ $\ (1<p<\infty ).$}
\author{Vatan KARAKAYA$^{\ast }$}
\address{Department of Mathematical Engineering, Yildiz Technical
University, Esenler, \.{I}stanbul, Turkey.}
\email{vkkaya@yahoo.com; vkkaya@yildiz.edu.tr}
\author{Manaf Dzh. MANAFOV}
\address{ADIYAMAN UNIVERSITY, Faculty of Arts and Sciences\\
Department of Mathematics, Ad\i yaman-TURKEY}
\email{mmanafov@adiyaman.edu.tr}
\author{Necip \c{S}\.{I}M\c{S}EK}
\address{\.{I}STANBUL COMMERCE UNIVERSITY, Faculty of Arts and Sciences\\
Department of Mathematics, \.{I}stanbul-TURKEY}
\email{necsimsek@yahoo.com}
\subjclass[2000]{ 47A10; 47B37;15A18}
\keywords{Spectrum of an operator, the sequence spaces $\ell _{p}$\ and $%
bv_{p}$, symmetric tri-band matrix.}

\begin{abstract}
In general, it is well known the behaviors of the symmetric tri-band
matrices on the Hilbert spaces. But the symmetric tri-band matrices have
different the behavior on the Banach spaces. The main purpose of this work
is to determine the fine spectra of the operator $U(s,r,s)$ defined by
symmetric tri-band matrix over the sequence spaces $\ell _{p}$\ and $bv_{p}$.
\end{abstract}

\maketitle

\section{Introduction}

As it is well known, the matrices play an important role in operator theory.
The spectrum of an operator generalizes the notion of eigenvalues for
matrices. In the calculation of the spectrum of an operator over a Banach
space, we mostly deal with three disjoint parts of the spectrum, which are
the point spectrum, the continuous spectrum and the residual spectrum. The
determination of these three parts of the spectrum of an operator is called
the fine spectra.

Over the years and different names the spectrum and fine spectra of linear
operators defined by some particular limitation matrices over some sequence
spaces have been studied.

In the existing literature, there are many papers concerning the spectrum
and the fine spectra of an operator over different sequence spaces. For
example; Gonzales\cite{gon}$,$ Akhmedov and Ba\c{s}ar\cite{akh2} computed
the fine spectra of \ the Ces\'{a}ro operator over the sequence spaces $\ell
_{p}$ and $c_{0}$ respectively$.$ Also Reade\cite{rea} and Okutoyi\cite{oku}
examined the spectrum of the Ces\'{a}ro operator over the spaces $c_{0}$ and 
$bv$ respectively. Later on Akhmedov and Ba\c{s}ar\cite{akh1,akh3} also
determined the spectrum of the Ces\'{a}ro operator and the fine spectrum of
the difference operator $\Delta $ over the sequence space $bv_{p}\left(
1<p<\infty \right) .$ Also, in \cite{wen}, Wenger studied the fine spectra
of H\"{o}lder summability operators over the space $c$ and Rhoades\cite{rho}
extended this result to the weighed mean method.

Recently, Karakaya and Altun$\left( \text{\cite{VM}, \cite{vkkalt}}\right) $%
computed respectively the fine spectras of the upper triangular double-band
matrices and the lacunary matrices as an operator over the sequence spaces $%
c_{0}$ and $c$, which are defined below.

Further informations on the spectrum and fine spectra of different operators
over some sequence spaces can be found in the list of references \cite{bil1}%
, \cite{bil},\cite{cos}, \cite{fur}.

The main purpose of our work is to determine the fine spectra of the
operator for which the corresponding matrix is the symmetrical tri-band
matrix $U(s,r,s)$ over the sequence spaces $\ell _{p}$ and $bv_{p}$ $\left(
1<p<\infty \right) .$

\section{Preliminaries and Notations}

Let $X$ and $Y$ be Banach spaces and $T:X\rightarrow Y$ be a bounded linear
operator. By $R(T)$, we denote the range of $T$, i.e., 
\begin{equation*}
R(T)=\{y\in Y:y=Tx;x\in X\}.
\end{equation*}%
By $B(X)$, we denote the set of all bounded linear operators on $X$ into
itself. If X is any Banach space and $T\in B(X),$ then the \emph{adjoint} $%
T^{\ast }$ of $T$ is a bounded linear operator on the dual $X^{\ast }$ of $X$
defined by $(T^{\ast }\phi )(x)=\phi (Tx)$ for all $\phi \in X^{\ast }$ and $%
x\in X$. Let $X\neq \{\theta \}$ be a complex normed space and $T:\mathcal{D}%
(T)\rightarrow X$ be a linear operator with domain $\mathcal{D}(T)\subset X$%
. By $T$ , we associate the operator 
\begin{equation*}
T_{\lambda }=T-\lambda I;
\end{equation*}%
where $\lambda $ is a complex number and $I$ is the identity operator on $%
\mathcal{D}(T)$. If $T_{\lambda }$ has an inverse, which is linear, we
denote it by $T_{\lambda }^{-1}$, that is 
\begin{equation*}
T_{\lambda }^{-1}=(T-\lambda I)^{-1},
\end{equation*}%
and it is called to be the \emph{resolvent operator} of $T$ . Many
properties of $T_{\lambda }$ and $T_{\lambda }^{-1}$ depend on $\lambda $,
and spectral theory is concerned with those properties. For instance, we
shall interest in the set of all $\lambda $ in the complex plane such that $%
T_{\lambda }^{-1}$ exists. Boundedness of $T_{\lambda }^{-1}$ is another
property that will be essential. We shall also ask for what all $\lambda $
in the domain of $T_{\lambda }^{-1}$ is dense in $X$. For investigation of $%
T $ , $T_{\lambda }$ and $T_{\lambda }^{-1}$ , we need some basic concepts
in spectral theory which are given as follows (see \cite{kre}, pp. 370-371):

Let $X \neq \{\theta\}$ be a complex normed space and $T : \mathcal{D}(T)
\rightarrow X$ be a linear operator with domain $\mathcal{D}(T) \subset X$.
A \emph{regular value} $\lambda$ of $T$ is a complex number such that 
\newline
(\textbf{R1}) $T_\lambda^{-1}$ exists, \newline
(\textbf{R2}) $T_\lambda^{-1}$ is bounded, \newline
(\textbf{R3}) $T_\lambda^{-1}$ is defined on a set which is dense in $X$.

The \emph{resolvent set} $\rho (T)$ of $T$ is the set of all regular values $%
\lambda $ of $T$ . Its complement $\sigma (T)=\mathbb{C}\backslash \rho (T)$
in the complex plane $\mathbb{C}$ is called the \emph{spectrum} of $T$ .
Furthermore, the spectrum $\sigma (T)$ is partitioned into three disjoint
sets as follows: The \emph{point spectrum} $\sigma _{p}(T)$ is the set such
that $T_{\lambda }^{-1}$ does not exist. A $\lambda \in \sigma _{p}(T)$ is
called an \emph{eigenvalue} of $T$ . The \emph{continuous spectrum} $\sigma
_{c}(T)$ is the set such that $T_{\lambda }^{-1}$ exists and satisfies (R3)
but not (R2). The \emph{residual spectrum} $\sigma _{r}(T)$ is the set such
that $T_{\lambda }^{-1}$ exists but not satisfy (R3).

We shall write $\ell _{\infty }$, $c$ and $c_{0}$ for the spaces of all
bounded, convergent and null sequences, respectively. The sequence spaces $%
\ell _{p}$ and $bv_{p}$ are defined by 
\begin{eqnarray}
\ell _{p} &=&\{x\in \omega :\sum_{k}|x_{k}|^{p}<\infty \}  \notag \\
bv_{p} &=&\{x\in \omega :\sum_{k}|x_{k}-x_{k+1}|^{p}<\infty \}.  \notag
\end{eqnarray}%
Let $\mu $ and $\gamma $ be two sequence spaces and $A=(a_{nk})$ be an
infinite matrix of real or complex numbers $a_{nk}$ , where $n$, $k\in 
\mathbb{N=}\left\{ 0,1,2,...\right\} $. Then, we say that $A$ defines a
matrix mapping from $\mu $ into $\gamma $, and we denote it by writing $%
A:\mu \rightarrow \gamma $, if for every sequence $x=(x_{k})\in \mu $ the
sequence $Ax=\{(Ax)_{n}\}$, the $A$-transform of $x$, is in $\gamma $; where 
\begin{equation}
(Ax)_{n}=\sum_{k}a_{nk}x_{k}\qquad (n\in \mathbb{N}).  \label{1}
\end{equation}%
By $(\mu :\gamma )$, we denote the class of all matrices $A$ such that $%
A:\mu \rightarrow \gamma $. Thus, $A\in (\mu :\gamma )$ if and only if the
series on the right side of (\ref{1}) converges for each $n\in \mathbb{N}$
and every $x\in \mu $, and we have $Ax=\{(Ax)_{n}\}_{n\in \mathbb{N}}\in
\gamma $ for all $x\in \mu $.

The symmetrical tri-band matrix used in our work is of the following form: 
\begin{equation*}
U(s,r,s)=\left[ 
\begin{array}{ccccc}
r & s & 0 & 0 & ... \\ 
s & r & s & 0 & ... \\ 
0 & s & r & s & ... \\ 
0 & 0 & s & r & ... \\ 
. & . & . & . & \ddots%
\end{array}%
\right]
\end{equation*}

Now let us give some the lemmas which we need in sequel

\label{Lem.1.1}

\begin{lemma}
Define the sets $D_{\infty }$ and $D_{q}$ by%
\begin{equation*}
D_{\infty }=\left\{ x=(x_{k})\in w:\sup_{k\in \mathbb{N}}\left\vert
\dsum\limits_{j=k}^{\infty }x_{j}\right\vert <\infty \right\}
\end{equation*}%
and%
\begin{equation*}
D_{q}=\left\{ x=(x_{k})\in w:\dsum\limits_{k}\left\vert
\dsum\limits_{j=k}^{\infty }x_{j}\right\vert ^{q}<\infty \right\} ,\text{ }%
(1<q<\infty ).
\end{equation*}%
Then, the sets $D_{\infty }$ and $D_{q}$ are the Banach spaces with the norms%
\begin{equation*}
\left\Vert a\right\Vert _{D_{\infty }}=\sup_{k\in \mathbb{N}}\left\vert
\dsum\limits_{j=k}^{\infty }a_{j}\right\vert
\end{equation*}%
and%
\begin{equation*}
\left\Vert a\right\Vert _{D_{q}}=\left( \dsum\limits_{k}\left\vert
\dsum\limits_{j=k}^{\infty }a_{j}\right\vert ^{q}\right) ^{1/q}.
\end{equation*}%
Additionally,$(i)$ $D_{\infty }$ is isometrically isomorphic to $%
bv_{1}^{\ast }$, \cite[Theorem 3.3]{Imaminezhad-Miri}

$(ii)$ $D_{q}$ is isometrically isomorphic to $bv_{p}^{\ast }$, \cite[%
Theorem 2.3]{Akhmedov-Basar-4}
\end{lemma}

The basis of the space $bv_{p}$ is also constructed and is given by the
following lemma:

\begin{lemma}
\label{Lem.1.2} $($\cite{Basar-Altay}, Theorem 3.1$)$ Define the sequence $%
b^{(k)}=\left\{ b_{n}^{(k)}\right\} _{n\in \mathbb{N}}$ of the elements of
the space $bv_{p}$ for every fixed $k\in \mathbb{N}$ by%
\begin{equation*}
b_{n}^{(k)}=\left\{ 
\begin{array}{cc}
0, & (n<k) \\ 
1, & (n\geq k)%
\end{array}%
\right\} \text{ \ \ for all }n\in \mathbb{N}.
\end{equation*}%
Then the sequence $\left\{ b_{n}^{(k)}\right\} _{n\in \mathbb{N}}$ is a
basis for the space $bv_{p}$ and any $x\in bv_{p}$\ has a unique
representation of the form%
\begin{equation*}
x=\dsum\limits_{k}\lambda _{k}b^{(k)}
\end{equation*}%
where $\lambda _{k}=x_{k}-x_{k-1}$ for all $k\in \mathbb{N}.$
\end{lemma}

\begin{lemma}
\label{Lem.1.3} $($\cite{Choud-Nanda}, p.253, Theorem 34.16$)$ The matrix $%
A=(a_{nk})$ gives rise to a bounded linear operator $T\in B(\ell _{1})$ from 
$\ell _{1}$ to itself if and only if the supremum of $\ell _{1}$\ norms of
the columns of $A$ is bounded.
\end{lemma}

\begin{lemma}
\label{Lem.1.4} $($\cite{Choud-Nanda}, p.245, Theorem 34.3$)$ The matrix $%
A=(a_{nk})$ gives rise to a bounded linear operator $T\in B(\ell _{\infty })$
from $\ell _{\infty }$ to itself if and only if the supremum of $\ell _{1}$\
norms of the rows of $A$ is bounded.
\end{lemma}

\begin{lemma}
\label{Lem.1.5} $($\cite{Choud-Nanda}, p.254, Theorem 34.18$)$ Let $%
1<p<\infty $ and $A\in (l_{\infty },l_{\infty })\cap (l_{1},l_{1})$. Then $%
A\in (l_{p},l_{p}).$
\end{lemma}

\begin{corollary}
\label{Cor.1.6} Let $\mu \in \left\{ l_{p},bv_{p}\right\} $ $(1<p<\infty )$. 
$U(s,r,s):\mu \rightarrow \mu $ is a bounded linear operator and $\left\Vert
U(s,r,s)\right\Vert _{(\mu ,\mu )}=2|s|+|r|$.
\end{corollary}

\section{The Spectrum Of The Operator $U(s,r,s)$ On The Sequence Space $\ell
_{p}$, $(1<p<\infty )$.}

In this section, the fine spectrum of the second order difference operator $%
U(s,r,s)$ over the sequence space $\ell _{p}$, $(1<p<\infty )$ have been
examined. We begin with a theorem concerning the bounded linearity of the
operator $U(s,r,s)$ acting on the sequence space $\ell _{p}$, $(1<p<\infty
). $

\begin{theorem}
\label{Theo.2.1} $U(s,r,s):\ell _{p}\rightarrow \ell _{p}$ is a bounded
linear operator satisfying the inequalities 
\begin{equation*}
\left( |r|^{p}+2|s|^{p}\right) ^{\frac{1}{p}}\leq \left\Vert
U(s,r,s)\right\Vert _{\ell _{p}}\leq 2|s|+|r|.
\end{equation*}
\end{theorem}

\begin{proof}
The linearity of $U(s,r,s)$ is trivial and so it is omitted. Let us take $%
e^{(1)}=(0,1,0,...)$ in $\ell _{p}.$ Then $U(s,r,s)e^{(1)}=(s,r,s,0,0,...)$
and observe that

\begin{equation*}
\left\Vert U(s,r,s)e^{(1)}\right\Vert _{\ell _{p}}=\left(
|r|^{p}+2|s|^{p}\right) ^{\frac{1}{p}}\leq \left\Vert U(s,r,s)\right\Vert
_{\ell _{p}}\left\Vert e^{(1)}\right\Vert _{\ell _{p}}
\end{equation*}%
which gives the fact that

\begin{equation}
\left( |r|^{p}+2|s|^{p}\right) ^{\frac{1}{p}}\leq \left\Vert
U(s,r,s)\right\Vert _{\ell _{p}}  \label{2.1}
\end{equation}%
for any $p>1$. Now take any $x=(x_{k})\in \ell _{p}$ such that $\left\Vert
x\right\Vert =1$ . Then, using Minkowski's inequality and taking $x_{-1}=0$,
we have

\begin{eqnarray*}
\left\Vert U(s,r,s)x\right\Vert _{\ell _{p}} &=&\left(
\sum\limits_{k=0}^{\infty }\left\vert sx_{k-1}+rx_{k}+sx_{k+1}\right\vert
^{p}\right) ^{\frac{1}{p}} \\
&\leq &\left( \sum\limits_{k=0}^{\infty }\left\vert sx_{k-1}\right\vert
^{p}\right) ^{\frac{1}{p}}+\left( \sum\limits_{k=0}^{\infty }\left\vert
rx_{k}\right\vert ^{p}\right) ^{\frac{1}{p}}+\left(
\sum\limits_{k=0}^{\infty }\left\vert sx_{k+1}\right\vert ^{p}\right) ^{%
\frac{1}{p}} \\
&=&\left( |r|+2|s|\right) \left\Vert x\right\Vert _{\ell _{p}},
\end{eqnarray*}%
which gives

\begin{equation}
\left\Vert U(s,r,s)\right\Vert _{\ell _{p}}\leq 2|s|+|r|.  \label{2.2}
\end{equation}%
Combining the inequalities (\ref{2.1}) and (\ref{2.2}) we complete the proof.
\end{proof}

\begin{theorem}
\label{Theo.2.2}$\sigma \left( U(s,r,s),\ell _{p}\right) =[r-2s,r+2s].$
\end{theorem}

\begin{proof}
First, we prove that $\left( U(s,r,s)-\lambda I\right) ^{-1}$ exists and in $%
B(\ell _{p})$ for%
\begin{equation*}
\lambda \notin \left\{ \lambda \in C:\lambda =r+2s.\cos \theta ,\theta \in
\lbrack 0,2\pi ]\right\}
\end{equation*}%
and next that the operator $\left( U(s,r,s)-\lambda I\right) $ is not
invertible for%
\begin{equation*}
\lambda \in \ \left\{ \lambda \in C:\lambda =r+2s.\cos \theta ,\theta \in
\lbrack 0,2\pi ]\right\} .
\end{equation*}%
Let $\lambda \notin \sigma \left( U(s,r,s),\ell _{p}\right) $. Let $\alpha
_{1}$ and $\alpha _{2}$ be the roots of the polynomial%
\begin{equation*}
P(x)=sx^{2}+(r-\lambda )x+s,\text{with }\left\vert \alpha _{2}\right\vert
>1>\left\vert \alpha _{1}\right\vert .
\end{equation*}%
Solving the system of equations%
\begin{eqnarray}
(r-\lambda )x_{1}+sx_{2} &=&y_{1}  \label{2.33} \\
sx_{1}+(r-\lambda )x_{2}+sx_{3} &=&y_{2}  \notag \\
sx_{2}+(r-\lambda )x_{3}+sx_{4} &=&y_{3}  \notag \\
&&\cdots  \notag
\end{eqnarray}%
for $x=(x_{k})$ in terms of $y=(y_{k})$ gives the matrix of $\left(
U(s,r,s)-\lambda I\right) ^{-1}.$ This is a non-homogenous linear recurrence
relation. Using the fact that $x,y\in \ell _{p}$, for (\ref{2.33}) we can
reach to a solution with generating functions (see \cite%
{Graham-Knuth-Patashnik}). This solution can be given by%
\begin{equation}
x_{k}=\frac{1}{s(\alpha _{1}^{2}-1)}\dsum\limits_{n=0}^{\infty }t_{kn}y_{n},
\label{2.4}
\end{equation}%
where%
\begin{equation*}
t_{kn}=\left\{ 
\begin{array}{cc}
\alpha _{1}^{k+1-n}-\alpha _{1}^{k+3-n}; & \text{if }k\geq n, \\ 
\alpha _{1}^{n+1-k}-\alpha _{1}^{n+3-k}; & \text{if }k<n.%
\end{array}%
\right. .
\end{equation*}%
{\small Thus, we obtain that%
\begin{equation*}
\left\Vert \left( U(s,r,s)-\lambda I\right) ^{-1}\right\Vert _{\left( \ell
_{1},\ell _{1}\right) }=\sup_{k}\sum\limits_{n=k}^{\infty }|x_{k}|\leq
\sup_{k}\left( \left\vert \alpha _{1}\right\vert +\left\vert \alpha
_{1}\right\vert ^{2k+3}\right) \dsum\limits_{n=0}^{k}\left\vert \alpha
_{1}\right\vert ^{n}<\infty ,
\end{equation*}%
}i.e. $\left( U(s,r,s)-\lambda I\right) ^{-1}\in \left( \ell _{1},\ell
_{1}\right) .$ {\small Similarly%
\begin{equation*}
\left\Vert \left( U(s,r,s)-\lambda I\right) ^{-1}\right\Vert _{\left( \ell
_{\infty },\ell _{\infty }\right) }<\infty .
\end{equation*}%
}By Lemma \ref{Lem.1.5}, we have $\left( U(s,r,s)-\lambda I\right) ^{-1}\in
(l_{p},l_{p})$. This shows that $\sigma \left( U(s,r,s),\ell _{p}\right)
\subseteq \left\{ \lambda \in C:\lambda =r+2s.\cos \theta ,\theta \in
\lbrack 0,2\pi ]\right\} .$

Let $\lambda \in \sigma \left( U(s,r,s),\ell _{p}\right) $ and $\lambda \neq
r$. Then $\left( U(s,r,s)-\lambda I\right) ^{-1}$ exists but $%
y=(1,0,0,...)\in l_{p}$ and $x=\left( x_{k}\right) $ not in $l_{p}$ , hence $%
\left\vert \alpha _{2}\right\vert >1>\left\vert \alpha _{1}\right\vert $ is
not satisfied, i.e. $\left( U(s,r,s)-\lambda I\right) ^{-1}$ is not in $%
B(\ell _{p}).$ If $\lambda =r$, then $U(s,r,s)-\lambda I=U(s,0,s)$. Since $%
U(s,0,s)x=\theta $ implies $x\neq \theta =(0,0,0,...)$, $U(s,r,s):\ell
_{p}\rightarrow \ell _{p}$ is not invertible. This shows that $\left\{
\lambda \in C:\lambda =r+2s.\cos \theta ,\theta \in \lbrack 0,2\pi ]\right\}
\subseteq \sigma \left( U(s,r,s),\ell _{p}\right) .$ On the other hand, $%
\alpha _{1}.\alpha _{2}=1$, $\left\vert \alpha _{2}\right\vert >1>\left\vert
\alpha _{1}\right\vert $ is not satisfied means, the roots can be only of
the form%
\begin{equation*}
\alpha _{1}=\frac{1}{\alpha _{2}}=e^{i\theta }
\end{equation*}%
for some $\theta \in \lbrack 0,2\pi )$. Then $\frac{\lambda -r}{s}=\alpha
_{1}+\alpha _{2}=e^{i\theta }+e^{-i\theta }=2\cos \theta $. Hence $\lambda
=r+2s.\cos \theta $, which means $\lambda $ can be only on the line segment $%
\left[ r-2s,r+2s\right] $. This completes the proof.
\end{proof}

We should remark that the index $p$ has different meanings in the notation
of the spaces $\ell _{p}$, $\ell _{p}^{\ast }\simeq \ell _{q}$ with $%
p^{-1}+q^{-1}=1$ and the point spectrums $\sigma _{p}\left( U(s,r,s),\ell
_{p}\right) $, $\sigma _{p}\left( U^{\ast }(s,r,s),\ell _{q}\right) $ which
occur in the following theorems.

\begin{theorem}
\label{Theo.2.3} $\sigma _{p}\left( U(s,r,s),\ell _{p}\right) =\varnothing .$
\end{theorem}

\begin{proof}
Let $\lambda $ be an eigenvalue of the operator $U(s,r,s)$. An eigenvector $%
x=(x_{1},x_{2},...)\in \ell _{p}$ corresponding to this eigenvalue satisfies
the lineer system of equations%
\begin{eqnarray}
rx_{1}+sx_{2} &=&\lambda x_{1} \\
sx_{1}+rx_{2}+sx_{3} &=&\lambda x_{2}  \notag \\
sx_{2}+rx_{3}+sx_{4} &=&\lambda x_{3}  \notag \\
&&\cdots  \notag
\end{eqnarray}

If $x_{1}=0$, then $x_{k}=0$ for all $k\in \mathbb{N}$. Hence $x_{1}\neq 0$.
Then the system of equations turn into the linear homogeneous recurrence
relation%
\begin{equation*}
x_{k+2}+qx_{k+1}+x_{k}=0\text{, for }k\geq 1\text{,}
\end{equation*}

where $q=\frac{r-\lambda }{s}$. The characteristic polynomial of the
recurrence relation is%
\begin{equation*}
x^{2}+qx+1=0.
\end{equation*}

There are two cases here.

\textbf{Case 1.} $\left\vert q\right\vert =2.$

Then characteristic polynomial has one root: $\alpha =\left\{ 
\begin{array}{cc}
1, & \text{if }q=-2 \\ 
-1, & \text{if }q=2%
\end{array}%
\right. .$ Hence, the solution of the recurrence relation is of the form%
\begin{equation*}
x_{n}=\left\{ 
\begin{array}{cc}
nx_{1}\text{ \ \ }; & \text{if }q=-2, \\ 
\left( -1\right) ^{n+1}nx_{1}\text{ \ }; & \text{if }q=2.%
\end{array}%
\right.
\end{equation*}

This means $\left( x_{n}\right) \notin \ell _{p}$. So, we conclude that
there is no eigenvalue in this case.

\textbf{Case 2.} $\left\vert q\right\vert \neq 2.$

Then characteristic polynomial has two distinct roots $\left\vert \alpha
_{1}\right\vert \neq 1$ and $\left\vert \alpha _{2}\right\vert \neq 1$ with $%
\alpha _{1}.\alpha _{2}=1$. Let $\left\vert \alpha _{2}\right\vert
>1>\left\vert \alpha _{1}\right\vert $. The solution of the recurrence
relation is of the form%
\begin{equation*}
x_{n}=A\left( \alpha _{2}\right) ^{n}+B\left( \alpha _{1}\right) ^{n}\text{.}
\end{equation*}

Using the fact that $qx_{1}+x_{2}=0$, we get $A=\frac{1}{\alpha _{2}-\alpha
_{1}}x_{1}$, $B=\frac{1}{\alpha _{1}-\alpha _{2}}x_{1}$. So we have%
\begin{equation*}
x_{n}=\frac{\left( \alpha _{2}\right) ^{n}-\left( \alpha _{1}\right) ^{n}}{%
\alpha _{2}-\alpha _{1}}x_{1}\text{.}
\end{equation*}

Again we have $\left( x_{n}\right) \notin \ell _{p}$. Hence there is no
eigenvalue also in this case.
\end{proof}

\begin{theorem}
\label{Theo.2.4} $\sigma _{p}\left( U^{\ast }(s,r,s),\ell _{p}^{\ast
}\right) =\varnothing .$
\end{theorem}

\begin{proof}
Since $U^{\ast }(s,r,s)=U^{t}(s,r,s)=U(s,r,s)$ from the Theorem \ref%
{Theo.2.3}, the proof is obtained easily.
\end{proof}

\begin{corollary}
\label{Cor.2.5} $\sigma _{r}\left( U(s,r,s),\ell _{p}\right) =\varnothing $.
\end{corollary}

\begin{theorem}
\label{Theo.2.6} $\sigma _{c}\left( U(s,r,s),\ell _{p}\right) =\left[
r-2s,r+2s\right] $
\end{theorem}

\begin{proof}
Since $\sigma _{p}\left( U(s,r,s),\ell _{p}\right) =\sigma _{r}\left(
U(s,r,s),\ell _{p}\right) =\varnothing $, $\sigma \left( U(s,r,s),\ell
_{p}\right) $ is the disjoint union of the parts $\sigma _{p}\left(
U(s,r,s),\ell _{p}\right) $, $\sigma _{r}\left( U(s,r,s),\ell _{p}\right) $
and $\sigma _{r}\left( U(s,r,s),\ell _{p}\right) $, we have $\sigma
_{c}\left( U(s,r,s),\ell _{p}\right) =\left[ r-2s,r+2s\right] .$
\end{proof}

\section{The Spectrum Of The Operator $U(s,r,s)$ On The Sequence Space $%
bv_{p}$, $(1<p<\infty )$}

In this section, the fine spectrum of the second order difference operator $%
U(s,r,s)$ over the sequence space $bv_{p}$, $(1<p<\infty )$ have been
examined. We begin with a theorem concerning the bounded linearity of the
operator $U(s,r,s)$ acting on the sequence space $bv_{p}$, $(1<p<\infty )$.

\begin{theorem}
\label{Theo.3.1} $U(s,r,s)\in B(bv_{p}).$
\end{theorem}

\begin{proof}
The linearity of the operator $U(s,r,s)$ is trivial and so it is omitted.
Let us take any $x=(x_{k})\in bv_{p}.$ Then, using Minkowski's inequality
and taking the negative indices $x_{-k}=0$, we have%
\begin{eqnarray*}
\left\Vert U(s,r,s)x\right\Vert _{bv_{p}} &=&\left(
\sum\limits_{k=0}^{\infty }\left\vert
sx_{k-1}+rx_{k}+sx_{k+1}-(sx_{k-2}+rx_{k-1}+sx_{k})\right\vert ^{p}\right) ^{%
\frac{1}{p}} \\
&\leq &\left( 2|s|+|r|\right) \left\Vert x\right\Vert _{bv_{p}}.
\end{eqnarray*}
\end{proof}

\begin{theorem}
\label{Theo.3.2} $\sigma (U(s,r,s),bv_{p})=\left[ r-2s,r+2s\right] .$
\end{theorem}

\begin{proof}
First, we prove that $(U(s,r,s)-\lambda I)^{-1}$ exists and is in $B(bv_{p})$
for $\lambda \notin \left[ r-2s,r+2s\right] $ and next that the operator $%
(U(s,r,s)-\lambda I)$ is not invertible for $\lambda \in \left[ r-2s,r+2s%
\right] .$

Let $\lambda \notin \left[ r-2s,r+2s\right] .$ Let $y=(y_{k})\in bv_{p}$.
This implies that $(y_{k}-y_{k-1})\in \ell _{p}$. Solving the equation $%
(U(s,r,s)-\lambda I)x=y,$ we find the matrix in the proof of Theorem \ref%
{Theo.2.2}. Then we obtain that 
\begin{equation*}
x_{k}-x_{k-1}=(U(s,r,s)-\lambda I)^{-1}(y_{k}-y_{k-1}).
\end{equation*}

Since $(U(s,r,s)-\lambda I)^{-1}\in (\ell _{p},\ell _{p})$ by Theorem \ref%
{Theo.2.2}, \ $(x_{k})\in bv_{p}$. This shows that $(U(s,r,s),bv_{p})%
\subseteq \left[ r-2s,r+2s\right] .$

Now, let $\lambda \in \left[ r-2s,r+2s\right] $ and $\lambda \neq r$. Then $%
(U(s,r,s)-\lambda I)^{-1}$ exists. Using (\ref{2.4}), it can be shown not
belong in $B(\ell _{p})$. If $\lambda =r$, then similar arguments as in the
proof of Theorem \ref{2.2} show that the operator $U(s,0,s):bv_{p}%
\rightarrow bv_{p}$ is not invertible. This shows that $\left[ r-2s,r+2s%
\right] \subseteq \sigma \left( U(s,r,s),bv_{p}\right) $. This completes the
proof.
\end{proof}

Since the spectrum and fine spectrum of the matrix $U(s,r,s)$ as an operator
on the sequence space $bv_{p}$ are similar to that of the space $\ell _{p}$
in Section 2, to avoid the repetition of the similar statements we give the
results in the following theorem without proof.

\begin{theorem}
\label{Theo.3.3}

\qquad $($i$)$ $\sigma _{p}\left( U(s,r,s),bv_{p}\right) =\varnothing ,$%
\newline
\qquad $($ii$)$ $\sigma _{p}\left( U^{\ast }(s,r,s),bv_{p}^{\ast }\right)
=\varnothing ,$ \newline
\qquad $($iii$)$ $\sigma _{r}\left( U(s,r,s),bv_{p}\right) =\varnothing ,$%
\newline
\qquad $($iv$)$ $\sigma _{c}\left( U(s,r,s),bv_{p}\right) =\left[ r-2s,r+2s%
\right] .$
\end{theorem}


\begin{thebibliography}{99}
\bibitem{akh1} A. M. Akhmedov, F. Ba\c{s}ar, \textit{On spectrum of the Ces%
\`{a}ro operator}, Proc. Inst. Math. Mech. Natl. Acad. Sci. Azerb. 19,
(2004) 3-8.

\bibitem{akh2} A. M. Akhmedov, F. Ba\c{s}ar, \textit{On the fine spectrum of
the Ces\`{a}ro operator in $c_{0}$}, Math. J. Ibaraki Univ. 36 (2004) 25-32.

\bibitem{akh3} A. M. Akhmedov, F. Basar, \textit{The fine spectra of the
difference operator $\Delta $ over the sequence space $bv_{p}$, $(1\leq
p<\infty )$}, Acta Math. Sin., 23 (2007) 1757-1768.

\bibitem{Akhmedov-Basar-4} B. Altay, F. Ba\c{s}ar, \textit{On the fine
spectrum of the generalized difference operator $B(r,s)$ over the sequence
spaces $c_{0}$ and $c$}, Internat. J. Math. Math. Sci. 18 (2005) 3005-3013.

\bibitem{vkkalt} M. Altun and V. Karakaya, \textit{Fine spectra of lacunary
matrices}, J. Commun. Math. Anal. 7, No. 1,(2009) 1-10.

\bibitem{Basar-Altay} F. Ba\c{s}ar and B. Altay, On the space of sequence of
p-bounded variation and related matrix mappings, Ukranian Math. J. 55(1)
(2003), 135-147.

\bibitem{bil1} H. Bilgi\c{c}, H. Furkan, \textit{On the fine spectrum of the
generalized diffrerence operator $B(r,s)$ over the sequence spaces $\ell
_{p} $ and $bv_{p}$}, Nonlinear Analysis 68 (2008) 499-506.

\bibitem{bil} H. Bilgi\c{c}, H. Furkan, \textit{On the fine spectrum of the
operator $B(r,s,t)$ over the sequence spaces $\ell _{1}$ and $bv$}, Math.
and Comp. Modelling 45 (2007) 883-891.

\bibitem{Choud-Nanda} B. Choudhary, S. Nanda, Functional Analysis with
Applications, John Wiley\&Sons Inc., New York, Chishester, Brisbane,
Toronto, Singapore, 1989.

\bibitem{cos} C. Co\c{s}kun, \textit{The spectra and fine spectra for p-Ces%
\`{a}ro operators}, Turkish J. Math., 21 (1997), 207-212.

\bibitem{fur} H. Furkan, H. Bilgic, K. Kayaduman, \textit{On the fine
spectrum of the generalized difference operator $B(r,s)$ over the sequence
spaces $\ell _{1}$ and $bv$}, Hokkaido Math. J. 35 (2006) 897-908.

\bibitem{Graham-Knuth-Patashnik} R.L. Graham, D.E. Knuth, O. Patashnik,
Concrete mathematics, Addison-Wesley Publishing company, 1989.

\bibitem{gol} S. Goldberg, \textit{Unbounded Linear Operators}, Dover
Publications, Inc., New York, 1985.

\bibitem{gon} M. Gonz\`{a}lez, \textit{The fine spectrum of the Ces\`{a}ro
operator in $\ell _{p}$ ($1<p<\infty $)}, Arch. Math. 44 (1985) 355-358.

\bibitem{Imaminezhad-Miri} M.\ Imam\i nezhed, M.R. Miri, The dual space of
the sequence space $bv_{p},(1\leq p<\infty )$, Acta Math. Univ. Comenian,
79(1)(2010), 143-149.

\bibitem{kre} E. Kreyszig, \textit{Introductory Functional Analysis with
Applications}, John Wiley \& Sons Inc., New York, 1978.

\bibitem{VM} V. Karakaya, M. altun, Fine spectra of upper triangular
double-band matrices, J. Comput. Appl. Math. 234(2010), 1387-1394.

\bibitem{oku} J. T. Okutoyi, \textit{On the spectrum of $C_{1}$ as an
operator on $bv$}, Commun. Fac. Sci. Univ. Ank. Ser. A1 41 (1992) 197-207.

\bibitem{rea} J. B. Reade, \textit{On the spectrum of the Ces\`{a}ro operator%
}, Bull. London Math. Soc. 17 (1985) 263-267.

\bibitem{rho} B. E. Rhoades, \textit{The fine spectra for weighted mean
operators}, Pacific J. Math. 104 (1) (1983) 219-230.

\bibitem{wen} R. B. Wenger, \textit{The fine spectra of H\"{o}lder
summability operators}, Indian J. Pure Appl. Math., 6, (1975)\ 695-712.
\end{thebibliography}
\end{document}